\documentclass[12pt]{amsart}
\usepackage{verbatim,amssymb,latexsym,amscd}
\usepackage[all]{xy}
\usepackage[colorlinks,linkcolor=blue,citecolor=blue,urlcolor=red]{hyperref}
\usepackage[T1]{fontenc}

\newcommand{\ie}{{\it i.e. }}
\newcommand{\cf}{{\it cf. }}
\newcommand{\eg}{{\it e.g. }}
\newcommand{\loccit}{{\it loc. cit. }}
\newcommand{\resp}{{\it resp. }}
\newcommand{\un}{\mathbf{1}}

\newcommand{\F}{\mathbf{F}}
\newcommand{\G}{\mathbb{G}}

\newcommand{\N}{\mathbf{N}}
\renewcommand{\P}{\mathbf{P}}
\newcommand{\Q}{\mathbf{Q}}

\newcommand{\bS}{\mathbf{S}}
\newcommand{\U}{\mathbf{U}}

\newcommand{\Z}{\mathbf{Z}}
\newcommand{\sA}{\mathcal{A}}
\newcommand{\sB}{\mathcal{B}}
\newcommand{\sC}{\mathcal{C}}
\newcommand{\sD}{\mathcal{D}}
\newcommand{\sE}{\mathcal{E}}

\newcommand{\sH}{\mathcal{H}}
\newcommand{\sI}{\mathcal{I}}
\newcommand{\bI}{\mathbb{I}}

\newcommand{\sK}{\mathcal{K}}

\newcommand{\sL}{\mathcal{L}}
\newcommand{\bM}{\mathbb{M}}
\newcommand{\sM}{\mathcal{M}}
\newcommand{\sN}{\mathcal{N}}

\newcommand{\bP}{\mathbb{P}}
\newcommand{\bQ}{\mathbb{Q}}
\newcommand{\sP}{\mathcal{P}}

\newcommand{\sV}{\mathcal{V}}

\newcommand{\bJ}{\mathbb{J}}

\newcommand{\fI}{\mathfrak{I}}
\newcommand{\fS}{\mathfrak{S}}

\newcommand{\Spec}{\operatorname{Spec}}
\newcommand{\GL}{\operatorname{GL}}

\newcommand{\Id}{\operatorname{Id}}
\newcommand{\Tr}{\operatorname{Tr}}
\newcommand{\Ker}{\operatorname{Ker}}
\newcommand{\Coker}{\operatorname{Coker}}
\newcommand{\Coim}{\operatorname{Coim}}
\newcommand{\IM}{\operatorname{Im}}

\newcommand{\abs}{{\operatorname{abs}}}

\newcommand{\Cont}{\operatorname{Cont}}

\newcommand{\tr}{{\operatorname{tr}}}

\newcommand{\End}{\operatorname{End}}

\newcommand{\Supp}{\operatorname{Supp}}

\newcommand{\cons}{{\operatorname{cons}}}

\newcommand{\Add}{{\operatorname{\bf Add}}}
\newcommand{\Ex}{{\operatorname{\bf Ex}}}

\newcommand{\Ab}{\operatorname{Ab}}

\newcommand{\rat}{{\operatorname{rat}}}

\newcommand{\num}{{\operatorname{num}}}

\newcommand{\rig}{{\operatorname{rig}}}

\newcommand{\tnil}{{\operatorname{tnil}}}

\newcommand{\ab}{{\operatorname{ab}}}

\renewcommand{\Vec}{\operatorname{\bf Vec}}

\newcommand{\Rep}{\operatorname{\bf Rep}}

\newcommand{\by}{\xrightarrow}

\newcommand{\iso}{\by{\sim}}

\newcommand{\inj}{\hookrightarrow}

\newcommand{\surj}{\rightarrow\!\!\!\!\!\rightarrow}
 
\newcommand{\colim}{\varinjlim}
\renewcommand{\lim}{\varprojlim}

\renewcommand{\qed}{\hfill $\Box$\medskip}

\renewcommand{\phi}{\varphi}
\renewcommand{\epsilon}{\varepsilon}

\newcommand{\sslash}{\mathbin{/\mkern-6mu/}}

\newcommand{\red}[1]{{\color{red} #1}}

\newcounter{spec}
\newenvironment{thlist}{\begin{list}{\rm{(\roman{spec})}}%
{\usecounter{spec}\labelwidth=20pt\itemindent=0pt\labelsep=10pt}}%
{\end{list}}%

\numberwithin{equation}{section}

\newtheorem{thm}{Theorem}[section]
\newtheorem{lemma}[thm]{Lemma}
\newtheorem{prop}[thm]{Proposition}
\newtheorem{cor}[thm]{Corollary}
\newtheorem{conj}[thm]{Conjecture}

\theoremstyle{definition}

\newtheorem{defn}[thm]{Definition}
\newtheorem{nota}[thm]{Notation}
\newtheorem{rk}[thm]{Remark}
\newtheorem{rks}[thm]{Remarks}

\newtheorem{ex}[thm]{Example}

\begin{document}
\title[Abelian tensor categories and Schur finiteness]{Universal rigid abelian tensor categories and Schur finiteness}
\author{Bruno Kahn}
\address{IMJ-PRG\\ Case 247\\4 place Jussieu\\
75252 Paris Cedex 05\\France}
\email{bruno.kahn@imj-prg.fr}
\date{\today}
\begin{abstract}
We study the construction of \cite{BVK} in more detail, especially in the case of Schur-finite rigid $\otimes$-categories. This leads to some groundwork on the ideal structure of rigid additive and abelian $\otimes$-categories.
\end{abstract}
\maketitle

\tableofcontents

\enlargethispage*{40pt}

\section{Introduction}

This note complements the results of \cite{BVK}, where we showed that any additive rigid $\otimes$-category maps to an abelian one in a universal way. We retain its definitions and notation, namely
\begin{itemize}
\item A \emph{$\otimes$-category} is an additive, symmetric, monoidal, unital category (with bilinear tensor product); a \emph{$\otimes$-functor} between $\otimes$-categories is a strong symmetric, monoidal, unital additive  functor.
\item $\Add^\otimes$ is the $2$-category of $\otimes$-categories, $\otimes$-functors and $\otimes$-natural isomorphisms.
\item $\Ex^\otimes$ is the $2$-category of abelian $\otimes$-categories, exact $\otimes$-functors and $\otimes$-natural isomorphisms.
\item $\Add^\rig$ and $\Ex^\rig$ are their $1$-full and $2$-full sub-$2$-categories of rigid categories.
\item For $\sC\in \Add^\otimes$, we write $Z(\sC):=\End_\sC(\un)$ (the \emph{centre} of $\sC$).
\end{itemize}

For $\sC\in \Add^\rig$, let $T(\sC)\in \Ex^\rig$ be the category of \cite[Th. 5.1]{BVK}: it defines a $2$-left adjoint to the forgetful $2$-functor $\Ex^\rig\to \Add^\rig$. As was observed in \loccit, the centre of $T(\sC)$ is in general not a field even if that of $\sC$ is. Previously, categories $\sA\in \Ex^\rig$ had been considered mainly when $Z(\sA)$ is a field; studying the general case now becomes indispensable. This is one of the tasks of this paper; another is to study the tensor ideals of objects  $\sC\in\Add^\rig$ in detail, and to relate them to those of the centre of $T(\sC)$. 

The main results are:

\enlargethispage*{50pt}

\subsection{Structure of rigid abelian $\otimes$-categories} Let $\sA\in \Ex^\rig$.

\begin{enumerate}
\item (Proposition \ref{p1})  $Z(\sA)$ is absolutely flat \cite[Ch. I, \S 2, ex. 17]{bbki} (von Neumann regular in another terminology, \cite[4.2]{weibel}).
\item (Theorem \ref{t1}, Remark \ref{r2} and Theorem \ref{p4}) There is a one-to-one correspondence between the ideals of $Z(\sA)$ and the Serre subcategories $\sI$ of $\sA$ stable under external tensor product. Moreover, for such a Serre subcategory, the localisation functor $\sA\to \sA\sslash\sI$ is full.
\end{enumerate}

Item (1) was found independently by Peter O'Sullivan. 

\subsection{$\otimes$-ideals} Let $\sC\in \Add^\otimes$. In Definition \ref{d5}, we introduce a ``Zariski'' topology on the set $\Spec^\otimes \sC$ of prime $\otimes$-ideals of $\sC$; it is spectral in the sense of Hochster \cite{hochster}. There is a spectral map \eqref{eq2b}  $\pi:\Spec^\otimes \sC\to \Spec Z(\sC)$. If $\sC\in \Add^\rig$, $\pi$ has a continuous closed section sending maximal ideals to maximal $\otimes$-ideals (Proposition \ref{ex1}); if moreover $\sC\in \Ex^\rig$, it has another ``minimal'' spectral section $\sigma$ (Proposition \ref{p17}). 

\subsection{Application to universal rigid abelian $\otimes$-categories} If $\sC\in \Add^\rig$, the local abelian envelopes of $\sC$ in the sense of Coulembier \cite{coul3} are classified by a (possibly empty) closed subset of $\Spec Z(T(\sC))$, where $T(\sC)\in \Ex^\rig$ is the universal category of \cite[Th. 5.1]{BVK} (Corollary \ref{c5}). Note that $\Spec Z(T(\sC))$ is profinite by the already quoted proposition \ref{p1}. 

\subsection{Schur-finite $\otimes$-categories} For $\sC$ as above, there is a canonical spectral map \eqref{eq2a} $\Spec Z(T(\sC))\to \Spec^\otimes\sC$. If $\sC$ is $\Q$-linear and Schur-finite, this map is a homeomorphism for the constructible topology on the right hand side (Corollary \ref{c4}; see \S \ref{s5.3} for the constructible topology). We also justify the claim of \cite[Rem. 6.6]{BVK} in Theorem \ref{t3} and get a refinement of  \cite[Prop. 8.5]{BVK}, restricted to motives of abelian type, in Corollary \ref{c9}.

\subsection{Free $\otimes$-categories} In Propositions \ref{p20}, \ref{p24} and in Therorem \ref{t6}, we describe $T(\sL_\Q)$ where $\sL_\Q$ is Deligne's free additive rigid category on one generator (\cite[(1.26)]{dm}, \cite[\S 10]{dS}) with $\Q$ coefficients.
\bigskip

I had planned to add further results on motives as in \cite{BVK}, but those are meagre and limited to Example \ref{ex3} and  Corollary \ref{c9}.

I thank Pierre Deligne for kindly explaining a misconception I had about \cite[Prop. 10.17]{dS}, and Ofer Gabber for suggesting that the tensor spectra of Section \ref{s5} might be spectral spaces, which clarified and simplified many of my proofs. I am especially indebted to Peter O'Sullivan, not only for his article \cite{os} from which I have taken many results, but also for enlightening correspondence during the preparation of this work. He had the intuition that the $2$-functor $T$ of \cite{BVK} is analogous to the process of rendering a commutative ring absolutely flat as in \cite{olivier}: this is vindicated by \cite[Ex. 5.5]{BVK} as well as Proposition \ref{p1}, Proposition \ref{p15} and especially Theorem \ref{c4} of this paper.

By Example \ref{ex2}, the present theory of tensor spectra extends that from commutative algebra, but this extension is limited: there is no localisation theory (see Lemma \ref{l9} c) and d)), and a Noetherian theory seems uninteresting (see Remark \ref{r5}). One thing I didn't try is to compare with Balmer's tensor triangular classification \cite{balmer} (like here, his tt spectra are spectral spaces). It should certainly be done. See also Krause \cite{krause2}.

\section{Kernels and ideals}\label{s2a}

This section  recalls well-known facts for later reference.

Let $F:\sC\to \sD$ be an additive functor between additive categories. We write 
\begin{align*}
\Ker_m(F)&= \{f\in Ar(\sC)\mid F(f)=0\}\\
\Ker_o(F)&= \{C\in Ob(\sC)\mid 1_C\in \Ker_m(F)\}= \{C\in Ob(\sC)\mid F(C)=0\}\\
\Ker_m^*(F)&=\{f\in Ar(\sC)\mid f \text{ factors through } C \text{ for some } C\in \Ker_o(F)\}.
\end{align*}

Let $I$ be a (two-sided) ideal of $\sC$ \cite[1.3]{AK2}. Then $I=\Ker_m(\sC\to \sC/I)$; we write occasionally $I_o$ and $I^*$ for $\Ker_o(\sC\to \sC/I)$ and $\Ker_m^*(\sC\to \sC/I)$.

\begin{lemma}\label{l7} Let $F:\sA\to \sB$ be an exact functor between abelian categories. Then $\Ker_o(F)$ is a Serre subcategory of $\sA$, and $\Ker_m^*(F)=\Ker_m(F)$.
\end{lemma}

\begin{proof} The first fact is obvious. For the second one, let $f:A\to B$ be in $\Ker_m(F)$. Factor $f$ as $A\surj C\inj B$, where $C=\IM f$. Then $C\in \Ker_o(F)$.
\end{proof}

In the situation of Lemma \ref{l7}, we shall abbreviate $\Ker_o F$ to $\Ker F$.

There will be a flurry of ideals of all sorts in the sequel. To distinguish them, we shall try and follow this notation:

\begin{itemize}
\item Ideals in commutative rings are denoted with capital italic letters.
\item In the special case of Boolean algebras, they are however denoted with gothic letters.
\item Serre $\otimes$-ideals in rigid abelian $\otimes$-categories (Definition \ref{d2} b)) are denoted with calligraphic letters. Serre localisations are denoted with double slashs $\sslash$, in order to distinguish them from quotients by ideals.
\item (Additive) $\otimes$-ideals in an rigid additive $\otimes$-category are denoted with blackboard letters.
\end{itemize}

\section{Absolutely flat rings}\label{s2}
\enlargethispage*{20pt}

In the sequel, we shall freely use the following equivalent properties for a commutative ring $R$ to be absolutely flat:

\begin{enumerate}
\item Any principal ideal is generated by an idempotent.
\item Any finitely generated submodule of a projective module is a direct summand.
\item For any $x\in R$, there exists $y$ such that $xyx=x$.
\end{enumerate}

\begin{lemma}\label{l5} Any finitely generated ideal $\fI$ of a Boolean algebra $B$  is principal. 
\end{lemma}

\begin{proof}\footnote{More directly: $B$ is absolutely flat.} We first show that $\fI$ is generated by orthogonal elements. Let $(e_1,\dots,e_n)$ be a set of generators of $\fI$. Assume that the statement is proven for $<e_1,\dots,e_{n-1}>$. We may then assume that they are mutually orthogonal. Define
\[f_i=e_ie_n,\; g_i = e_i(1+e_n)\; (1\le i\le n-1),\; h=(1+\sum e_j)e_n.\]

Clearly, $f_ig_j=0$ for $i\ne j$, $f_ig_i=0$, $g_ih=0$ and
\[f_ih =e_i(1+\sum e_j)e_n=(e_i+e_i)e_n=0\]
so these elements are mutually orthogonal. Finally, 
\[f_i+g_i=e_i\;  (1\le i\le n-1),\; h+\sum f_j=e_n\]
so they generate $\fI$.

(Thus, if we start from $n$ elements, we end up with at most  $2^n-1$ orthogonal elements.)

Now, if $e=\sum e_i$, we have $e_i=ee_i$ for all $i$, so $e$ generates $\fI$.
\end{proof}

Let $R$ be a commutative ring. The set $B(R)$ of idempotents of $R$ is in one-to-one correspondence with the open-closed (clopen) subsets of (the underlying topological space to) $\Spec R$. This gives $B(R)$ the structure of a
Boolean algebra for the addition $e\oplus e' =e+e'-2ee'$ and the multiplication $e\wedge e'= ee'$ corresponding to symmetric difference and intersection.

Let $I$ be an ideal of $R$. The set $B(I)$ of idempotents in $I$ is an ideal of $B(R)$. Conversely, to any ideal $\fI$ de $B(R)$, we may associate the ideal $I(\fI)$ of $R$ generated by $\fI$.

\begin{prop} \label{l4} The map $I\mapsto B(I)$ is left inverse to the map $\fI\mapsto I(\fI)$, and is a right inverse if and only if $R$ is absolutely flat.\footnote{I thank Kevin Coulembier for suggesting the ``only if'' part.} In particular, any ideal of an absolutely flat ring is generated by its idempotents.
\end{prop}

\begin{proof} We have obvious inclusions $\fI\subseteq B(I(\fI))$ and $I(B(I))\subseteq I$. For an ideal $\fI$ of $B(R)$, let $e\in B(I(\fI))$. Write $e=\sum r_ie_i$ with $r_i\in R$ and $e_i\in \fI$.  By Lemma \ref{l5}, there is $e'\in \fI$ such that $e_i=f_i e'$ for all $i$. Thus $e=re'$ for $r=\sum r_if_i$; but then $ee'=r{e'}^2=re'=e$, so $e\in \fI$ and $B(I(\fI))=\fI$.  

Assume now that $R$ is absolutely flat. For an ideal $I$ of $R$, let $x\in I$. Then $Rx =Re$ for some idempotent $e\in Rx\subseteq I$. This shows that $I(B(I))= I$. Conversely, this equality for $I=Rx$ with $x\in R$ implies that $Rx$ is generated by its idempotents. The same computation as in the beginning of the proof then shows that $x=xe$ for some idempotent $e\in Rx$, so that $Rx=Re$. This implies that $R$ is absolutely flat.
\end{proof}

The following lemma will be used in the proof of Proposition \ref{p20}.

\begin{lemma}\label{l13} Let $R$ be an absolutely flat $F$-algebra, where $F$ is a field, and assume that the composition $F\to R\to R/M$ is surjective for any $M\in \Spec R$. Then the rule $a\mapsto a\pmod{M}$ yields an isomorphism
\[\theta:R\iso \Cont(\Spec R,F) \]
where $\Cont$ denotes continuous functions.
\end{lemma}

\begin{proof} Write $X=\Spec R$. Let $a\in R$. For $f\in F$, the set 
\[\{M\in X\mid \theta(a)(M)=f\} = \{M\in X\mid \theta(a-f1)(M)=0\}=V(a-f1)\]
is closed, hence $\theta(a)$ is continuous and $\theta$ is well-defined. It is injective because $\bigcap_{M\in X} M=0$ since $R$ is absolutely flat. Finally, let us show its surjectivity: Let $\phi\in \Cont(X,F)$. Then $\phi(X)$ is finite since $X$ is compact (Hausdorff), which determines a partition of $X$ into the clopen subsets $\phi^{-1}(f)$ ($f\in \phi(X)$). Let $e_f\in R$ be the idempotent such that $V(e_f)=\phi^{-1}(f)$; then we have the ``partition'' $\phi=\theta\big(\sum_{f\in \phi(X)} f e_f\big)$.
\end{proof}

\section{More on rigid abelian $\otimes$-categories}

Let $\sA\in \Ex^\rig$.

\begin{defn}\label{d3} 
a) $\sA$ is \emph{connected} if $Z(\sA)$ is a field.\\
b) A \emph{Serre $\otimes$-ideal} $\sI$ of $\sA$ is a Serre subcategory of $\sA$ stable under external tensor product. We write $\sA\sslash\sI$ for the corresponding localisation (``Serre quotient''), in order to avoid confusion with the quotient by an (additive) $\otimes$-ideal.
\end{defn}

By  \cite[Rem. 2.10]{exandfaith}, a) is equivalent to $\sA$ being integral.

\subsection{Structure of $Z(\sA)$}
The following elaborates on \cite[Rem. 1.18]{dm}:

\begin{prop}\label{p1} The ring $Z=Z(\sA)$ is absolutely flat; the class $\U$ of subobjects of $\un$ is a set which is in one-to-one correspondence with the open-closed (clopen) subsets of (the underlying topological space to) $\Spec Z$. This correspondence is non-decreasing (for the inclusion relation in both sets), respects intersections and exchanges union in $\Spec Z$ with sum in $\U$.
\end{prop}

\begin{rk} This proposition justifies the terminology of Definition \ref{d3} a): it shows that as soon as $Z(\sA)$ is not a field, it contains a non-trivial idempotent which in turn yields a decomposition $\sA\simeq \sA_1\times \sA_2$ by \cite[Rem. 1.18]{dm}. See Theorems \ref{t1} and \ref{p4} for a generalisation. 
\end{rk}

To prove Proposition \ref{p1}, we need a lemma:

\begin{lemma}\label{l1} a) For $U,V\in \U$, we have $U\otimes V=U\cap V$.\\
b) For $U\in \U$, the decomposition
\begin{equation}\label{eq1}
\un=U\oplus U^\perp
\end{equation}
of \cite[Prop. 1.17]{dm} is unique.\\
c) For any such $U$, the  canonical isomorphism $\un\simeq \un^\vee$ identifies $U$ and $U^\vee$.\\
d) For any $x\in Z$, we have $\Ker x\oplus \IM x=\un$ and $x_{|\IM x}$ is invertible. In particular, $\IM x=(\Ker x)^\perp$.
\end{lemma}

\begin{proof} a) The case $U=V$ is contained in the proof of \cite[Prop. 1.17]{dm}. In general, the exactness of tensor product \cite[Prop. 1.16]{dm} gives an inclusion $U\otimes V\subseteq U\cap V$, and conversely
\[U\cap V= (U\cap V)\otimes (U\cap V)\subseteq U\otimes V. \]
b) Let $\un=U\oplus V$ be another decomposition. We have
\begin{multline*}
\un=(U\oplus U^\perp)\otimes (U\oplus V) = U\otimes U\oplus U\otimes V\oplus U^\perp\otimes U\oplus U^\perp\otimes V \\
= U\oplus U^\perp\cap V
\end{multline*}
by a). Hence $U^\perp = U^\perp\cap V=V$. 

c) Dualising the sequence $\un\surj U\inj \un$, we get a sequence $\un\surj U^\vee\inj \un$. If $U=\IM e$ for $e$ an idempotent of $Z(\sC)$, this identifies $U^\vee$ with $\IM(^t e)$. But  ${}^t e=e$ since $e\in Z(\sC)$.

d) By c) and \cite[Prop. 17]{dm}, we have
\[(\IM x)^\perp = \Ker(\un\to (\IM x)^\vee)=\Ker(\un\to \IM x) = \Ker x.\]

This shows that $x_{|\IM x}$ is mono. The dual reasoning gives an isomorphism
\[\un\iso \IM x\oplus \Coker x\]
which in turn shows that $x_{|\IM x}$ is epi. Thus it is an isomorphism, as claimed.
\end{proof} 

\begin{proof}[Proof of Proposition \ref{p1}] The clopen subsets of $\Spec Z$ are parametrised by the idempotents of $Z$. Let $e$ be such an idempotent: we associate to it $U(e)=\IM e\subseteq \un$. Conversely, if $U\subseteq\un$, let $e(U)$ be the idempotent with image $U$ given by the decomposition \eqref{eq1}. Let us show that these correspondences are inverse to each other:
\begin{itemize}
\item $U(e(U))=U$: this is trivial.
\item $e(U(e))=e$: this follows from Lemma \ref{l1} b).
\end{itemize}

Let us now show that $Z$ enjoys property (3) in the beginning of Section \ref{s2}. Let $x\in Z$: Lemma \ref{l1} d) allows us to choose $y$ such that $y_{|\Ker x}=0$ and $y_{|\IM x}= (x_{|\IM x})^{-1}$. Thus $Z$ is absolutely flat.

Finally, the claims about the ordered structures is clear from this correspondence.
\end{proof}

\begin{rk}\label{r3} In the sequel, the idempotent $e(U)$ associated to $U\in \U$, which was used in the above proof, will play an important rôle. We record its properties: for $U,V\in \U$:
\begin{itemize}
\item $e(U\cap V) = e(U) e(V)$;
\item $e(U+V) = e(U) + e(V) -e(U) e(V)$.
\end{itemize}
\end{rk}

\subsection{The trivial part of $\sA$}

\begin{defn} Let $\sC$ be an additive category. Given $C\in \sC$, we write $\sC(C)$ for the smallest strictly full subcategory of $\sC$ containing $C$ and closed under direct sums and direct summands.
\end{defn}

\begin{lemma} With the above notation, $\sC(C)$ is equivalent to a full subcategory of the category $\sP$ of finitely generated projective left $\End_\sC(C)$-modules, with equality if $\sC$ is pseudo-abelian.
\end{lemma}

\begin{proof} Let $R=\End_\sC(C)$. The preadditive subcategory of $\sC$ determined by $C$ is tautologically equivalent to that determined by $R$, and the full additive subcategory determined by the $C^n$ for $n\ge 0$ is equivalent to that of free finitely generated left $R$-modules. Since $\sP$ is the pseudo-abelian hull of the latter, the conclusion follows.
\end{proof}

\begin{prop}\label{p9} Let $\sA\in \Ex^\rig$. Then $\sA(\un)$  is a Serre subcategory of $\sA$, split in the sense of \cite[Def. 4.2]{BVK}. Moreover it is a $\otimes$-subcategory of $\sA$,  in which every object is self-dual.
\end{prop}

\begin{proof} We show successively:
\begin{enumerate}
\item Every epimorphism $f:\un^n\surj B$ in $\sA$ has a section (in particular, $B\in \sA(\un)$). 
\item Any $B$ as in (1) is a direct sum of $n$ direct summands of $\un$.
\item Every short exact sequence $0\to B'\to B\to B''\to 0$ in $\sA$ with $B\in \sA(\un)$ splits.
\end{enumerate}

(1) Consider the commutative diagram
\[\xymatrix{
\un^{n-1}\ar[d]^{i_n}\ar[dr]^g\\
\un^n\ar[r]^f\ar[d]^{p_n}& B\ar[d]^h\\
\un \ar[r]^{\bar f}&\bar B
}\]
where $i_n$ is the inclusion of the first $n-1$ summands, $p_n$ is the $n$-th projection and $\bar B=\Coker g$. By \cite[Prop. 1.17]{dm}, $\bar f$ has a section $s$ (unique by Lemma \ref{l1} b), but we don't care). Composing $s$ with a section $s_n$ of $p_n$, and then with $f$, we get a section $s'$ of $h$. Then $B=\IM g \oplus s'(\bar B)$. By induction on $n$, choose a section $s''$ of $g_1:\un^{n-1}\to \IM g$; then we get a section $\sigma$ of $f$ by $\sigma_{|\IM g}=i_ns''$ and $\sigma_{|s'(\bar B)}=s_ns$.

(2) follows from the proof of (1), by induction on $n$.

(3) follows formally from (1).

Item (3) shows that $\sA(\un)$ is a Serre subcategory of $\sA$ and is split, by \cite[Prop. 4.3 (3)]{BVK}. The claim on self-duality now follows from Lemma \ref{l1} c).
\end{proof}

\begin{rk} Conversely, the category of finitely generated projective modules over an absolutely flat commutative ring $R$ defines a split, rigid, self-dual $\otimes$-category $\sA(R)$ (compare property (2) at the beginning of Section \ref{s2}). 
This shows that $Z(\sA)$ can be any absolutely flat commutative ring. If we want to be fanciful, we can say that the $2$-functor $R\mapsto \sA(R)$ is $2$-left adjoint to the $2$-functor  $\sA\mapsto Z(\sA)$.\footnote{Interpreting monoids as categories with one object gives their category a structure of $2$-category. Given two parallel homomorphisms $f,g:M\rightrightarrows N$ of monoids, a natural transformation $f\Rightarrow g$ is an element $n\in N$ such that $f(m)n = ng(m)$ for all $m\in M$.}
\end{rk}

\subsection{The Serre $\otimes$-ideals of $\sA(\un)$} To such an ideal $\sI$, associate the set of idempotents $\fI(\sI)=\{e(U)\mid U\in \sI\cap \U\}$. This is an ideal of $B(Z(\sA))$, the Boolean algebra associated to $Z(\sA)$: equivalently, $\sI\cap \U$ is closed under sums and subobjects (see Lemma \ref{l1} a)). Conversely, to an ideal $\fI$ of $B(Z(\sA))$, associate the full additive subcategory $\sI(\fI)$ of $\sA(\un)$ generated by the $\IM e$ for $e\in \fI$: it is a Serre $\otimes$-ideal of $\sA(\un)$.

\begin{prop}\label{p10} The maps $\sI\mapsto \fI(\sI)$ and $\fI\mapsto \sI(\fI)$ are inverse to each other. They yield a bijective correspondence between  the Serre $\otimes$-ideals of $\sA(\un)$ and the ideals of $B(Z(\sA))$.
\end{prop}

\begin{proof} The inclusions $\sI\supseteq \sI(\fI(\sI))$ and $\fI(\sI(\fI))\supseteq \fI$ are tautological. Let $\sI$ be a $\otimes$-ideal of $\sA(\un)$, and let $A\in \sI$: by item (2) of the proof of Proposition \ref{p9}, we can write $A=\bigoplus A_i$ with $A_i\in \langle\un\rangle^\natural$, and all $A_i$ belong to $\sI$. This shows equality in the first inclusion. Let now $\fI$ be an ideal of $B(Z(\sA))$. Choose an orthogonal basis $(e_1,\dots,e_n)$ of $\fI$ as in Lemma \ref{l5}, and let $U_i=\IM e_i$: then $\sA(U_i,U_j)=0$ if $i\ne j$, and any object $A\in \sI(\fI)$ has a unique decomposition of the form $A\simeq \bigoplus_{i=1}^n U_i^{n_i}$; we have $A\in \U$ if and only if all $n_i$ are $\le 1$. Let $e\in \fI(\sI(\fI))$: writing $\IM e$ in this form, we find that $e=\sum n_ie_i$, hence $e\in \fI$ as desired. 
\end{proof}

\begin{rk}\label{r1} By Propositions \ref{l4}, \ref{p1} and \ref{p10}, there is so far a bijective correspondence between 
\begin{enumerate}
\item the ideals of $B(Z(\sA))$;
\item the ideals of $Z(\sA)$;
\item the subsets of $\U$ stable under sums and subobjects, \ie filters for the order relation opposite to inclusion. For simplicity, we shall call the latter \emph{cofilters} of $\U$.
\item the Serre $\otimes$-ideals of $\sA(\un)$.
\end{enumerate}
\end{rk}

\subsection{Supports}

\begin{defn}\label{d4} Let $f:A\to B$ be a morphism of $\sA$, and let $\tilde f:\un \to A^\vee\otimes B$ be its adjoint. The \emph{support of $f$} is
\[\Supp(f) =\Ker(\tilde f)^\perp\simeq \IM \tilde f.\]
For $A\in \sA$, we define $\Supp(A)=\Supp(1_A)$.\\ 
We set $e(f) = e(\Supp(f))$ and $e(A)=e(\Supp(A))$ (see Remark \ref{r3}).
\end{defn}

By definition, $\Supp(f)$ is the smallest subobject $U$ of $\un$ such that $f$ factors through $U \otimes B$, and $\Supp(A)$ is the smallest subobject $U$ of $\un$ such that $U^\perp \otimes A=0$. Thus we also have $\Supp(f)=\Supp(\IM(f))$. 



\begin{lemma}\label{l14} For any $A\in \sA$, one has $A=A\otimes \Supp(A)$.\qed
\end{lemma}

\begin{lemma}\label{l2} a) $\Supp(f\otimes g) = \Supp(f)\cap \Supp(g)$, hence also $e(f\otimes g) = e(f)e(g)$, for any morphisms $f,g$.\\
b) $\Supp(f\circ g)\subseteq \Supp(f)\cap \Supp(g)$, hence also $e(f)e(g)|e(f\circ g)$, for any composable morphisms $f,g$.\\
c) $f=f\circ e(f)$ for any $f$ with domain $\un$.
\end{lemma}

\begin{proof} a) follows from Lemma \ref{l1} a) and the fact that $\widetilde{f\otimes g} = \tilde f\otimes \tilde g$. In b), the inclusion $\Supp(f\circ g)\subseteq \Supp(f)$ is obvious, and the other inclusion can be seen dually. Finally, c) is easy.
\end{proof}

Lemma \ref{l2} a) allows us to give the right generalisation of \cite[Prop. 2.5 a) and b)]{BVK}:

\begin{prop}\label{p2} Given two morphisms $f,g$, one has $f\otimes g = 0$ if and only if $\Supp(f)\cap\Supp(g)=0$.\qed
\end{prop}

\begin{prop}\label{p3} Let $(*)\; 0\to A'\to A\to A''\to 0$ be a short exact sequence in $\sA$. Then\\
a) $\Supp(A) = \Supp(A')+\Supp(A'')$.\\
b) If $\Supp(A')\cap \Supp(A'')=0$, $(*)$ is split.
\end{prop}

\begin{proof}
a) Given a subobject $U$ of $\un$, $U^\perp \otimes A=0$ $\iff$ $U^\perp \otimes A'=0$ and $U^\perp \otimes A''=0$.

b) By the exactness of $\otimes$, we have a short exact sequence
\[0\to A'\otimes \Supp(A')\to A'\otimes \Supp(A')\to A''\otimes \Supp(A')\to 0\]
where, by Lemma \ref{l14}, $A'=A'\otimes \Supp(A')$ and $A''\otimes \Supp(A')=A''\otimes \Supp(A'')\otimes \Supp(A')=0$, the last equality by hypothesis. Thus
\[A'\iso A\otimes \Supp(A').\]

In the same way, we have an isomorphism
\[A\otimes \Supp(A'')\iso A''.\]

Since $\Supp(A)=\Supp(A')\oplus \Supp(A'')$ by a), we get an isomorphism
\[A\simeq A'\oplus A''\]
which splits $(*)$ by construction.
\end{proof}

\subsection{The Serre $\otimes$-ideals of $\sA$} We now want to add the latter as a fifth item to the list of Remark \ref{r1}. In view of the above, the most convenient is to compare them with the cofilters of $\U$.

Namely, to a Serre $\otimes$-ideal $\sI\subseteq \sA$, we associate $\Phi(\sI)=\sI\cap \U$; this is a cofilter of $\U$.  Conversely, to a cofilter $\Phi$ of $\U$, we associate the full subcategory $\sI(\Phi)=\{A\in \sA\mid \Supp(A)\in \Phi\}$: this is a Serre $\otimes$-ideal by Lemma \ref{l2} a) and Proposition \ref{p3} a).   


\begin{lemma}\label{p12} Let $\sI$ be a Serre $\otimes$-ideal of $\sA$. Then $A\in \sI$ $\iff$ $\Supp(A)\in \Phi(\sI)$.
\end{lemma}

\begin{proof} If $A\in \sI$, so do $A\otimes A^\vee$ and the image of $\eta:\un\to A\otimes A^\vee$. This proves $\Rightarrow$. Conversely, if $\Supp(A)\in \Phi(\sI)$, then $A=A\otimes \Supp(A)\in \sI$ (Lemma \ref{l14}). 
\end{proof}


\begin{thm} \label{t1} The maps $\sI\mapsto \Phi(\sI)$ and $\Phi\mapsto \sI(\Phi)$ are mutually inverse bijections; there is a $1-1$ correspondence between Serre $\otimes$-ideals of $\sA$ and ideals of $Z(\sA)$.
\end{thm}

\begin{proof} For a Serre $\otimes$-ideal $\sI$ of $\sA$, $\sI(\Phi(\sI))=\sI$ follows from Lemma \ref{p12}. Conversely, if $\Phi$ is a cofilter of $\U$, then, for $U\in \U$:
\[U\in \Phi(\sI(\Phi))\iff U\in \sI(\Phi)\iff \Supp(U)\in \Phi
\]
hence $\Phi(\sI(\Phi))=\Phi$ since $\Supp(U)=U$. This proves the first claim, and the second one then follows from Remark \ref{r1}.
\end{proof}

\begin{ex} If $\sA$ is connected, Theorem \ref{t1} says that $\sA$ has no proper Serre $\otimes$-ideals: we recover \cite[Prop. 1.19]{dm}.
\end{ex}

\begin{rk}\label{r2}
For later reference, let us specify the correspondence of Theorem \ref{t1}: in one direction it associates to a Serre $\otimes$-ideal $\sI$ the ideal $I(\sI)\subseteq Z(\sA)$ generated by the idempotents $e(A)$ for $A\in \sI$.  Conversely, if $I$ is an ideal of $Z(\sA)$, the full subcategory $\sI(I)$ of $\sA$ formed of those $A$ such that  $e(A) \in I$ 
is the desired Serre $\otimes$-ideal of $\sA$.
\end{rk}

\subsection{Centre and quotients}  Let $\sI$ be a Serre $\otimes$-ideal of $\sA$. By   \cite[Prop. 3.5]{BVK}, the exact localisation functor $F:\sA\to \sA\sslash\sI$ descends the tensor structure of $\sA$ to $\sA\sslash \sI$, giving it  the structure of an abelian rigid $\otimes$-category; moreover, the induced homomorphism $Z(F):Z(\sA)\to Z(\sA\sslash\sI)$ is surjective. The following is a complement to this proposition:
 
 \begin{thm}\label{p4} The kernel $\Ker_o(Z(F))$ of $Z(F)$ is equal to $I(\sI)$ (see Remark \ref{r2}), and the induced functor 
 \[Z(\sA)/I(\sI)\otimes_{Z(\sA)} \sA\to \sA\sslash\sI\]
is an equivalence of $\otimes$-categories; in particular, $F$ is \emph{full}.
\end{thm}

\begin{proof}  Let $f\in Z(\sA)$. If $e(f)\in I(\sI)$, clearly $F(f)=0$. Conversely, assume that $F(f)=0$. By definition, there are $\lambda_1,\dots,\lambda_n\in I(\sI)$ and $g_1,\dots,g_n\in Z(\sA)$ such that 
$f=\sum_i g_i\circ \lambda_i$. Let $I_0=\langle \lambda_1,\dots,\lambda_n\rangle$. As a finitely generated ideal of an absolutely flat ring, it is generated by an idempotent $e$; thus, writing $\lambda_i = \mu_i e$ for all $i$, we get $f=g\circ e$ with $g=\sum_i g_i\circ \mu_i$. Then $e(f)$ is a multiple of $e$, hence belongs to $I$.

Next, we prove the fullness of $F$. It suffices by rigidity to show that $\sA(\un,A)\allowbreak\to (\sA\sslash\sI)(\un,A)$ is surjective for any $A\in \sA$. Let $f\in (\sA\sslash\sI)(\un,A)$. By \cite[III.1]{gabriel}, $f$ may be represented by a morphism in $\sA$
\[\tilde f:\un'\to A'\]
where $\un'$ (\resp $A'$) is a subobject (\resp a quotient) of $\un$ (\resp $A$) such that $\un/\un'\in \sI$ (\resp $N\in \sI$, where $N=\Ker(A\to A')$). Let $\Sigma = \Supp(N)$. Write $\un=\Sigma\oplus \Sigma'$, and let $\un''=\un'\otimes \Sigma'$: then
\[\un\simeq (\un/\un')\oplus  \un'\otimes \Sigma \oplus \un''\]
where the first two summands belong to $\sI$, the second one because $\Sigma\in \sI$ by Lemma \ref{p12} and because $\sI$ is a Serre $\otimes$-ideal. Thus, up to replacing $\un'$ by $\un''$, we may assume that $\un'\cap \Sigma=\un'\otimes \Sigma=0$. By Proposition \ref{p3} b), the pull-back by $\tilde f$ of the extension $0\to N\to A\to A'\to 0$ then splits; this yields a lift $\tilde f':\un'\to A$ of $\tilde f$. Since $\un'$ is a direct summand of $\un$, $\tilde f'$ further extends to a morphism $\un\to A$, which still represents $f$.

To conclude, it remains to show that, for any $A,B\in \sA$, $\Ker(\sA(A,B)\allowbreak\to (\sA\sslash\sI)(A,B))\subseteq \Ker(\sA(A,B)\to  Z(\sA)/I(\sI)\otimes_{Z(\sA)} \sA(A,B))$: this follows from Lemma \ref{l7}.
\end{proof}

\begin{nota}\label{n1} Given an ideal $I$ of $Z(\sA)$, we simply write $\sA\sslash I$ for $\sA\sslash \sI(I)$ (see Remark \ref{r2}).
\end{nota}

Note that there is a $1-1$ correspondence between Serre $\otimes$-ideals of $\sA\sslash I$ and Serre $\otimes$-ideals of $\sA$ containing $\sI(I)$; by the above, we have $\sI\supseteq \sI(I)$ if and only if $I(\sI)\supseteq I$.

\begin{cor}\label{c1} There exists a conservative family of $\otimes$-functors from $\sA$ to connected rigid abelian $\otimes$-categories $\sA_M$.
\end{cor}

\begin{proof} Let $M$ run through the maximal ideals of $Z(\sA)$, and let $\sA_M := \sA\sslash M$. By Theorem \ref{p4}, $Z(\sA_M)$ is a field for all $M$. Let $A\in \sA$ be such that $A_M=0$ for all $M$. Then $e(A)\in M$ for all $\alpha$. But, in an absolutely flat ring, the intersection of all maximal ideals is $0$. Hence $e(A)=0$ and $A=0$.
\end{proof}

\begin{rks}
a) This corollary depends on the axiom of choice!\\
b) If $Z(\sA)$ is not Noetherian, the choice of all maximal ideals in the proof of Corollary \ref{c1} is redundant. For example, let $\sA$ be the $\otimes$-category of finitely generated projective modules over $R=\prod_p \F_p$. Then $\Spec R$ is much larger than the set of prime numbers (ultrafilters), but this set obviously suffices to give a conservative family.
\end{rks} 

\begin{cor}\label{c8} If $Z(\sA)$ is Noetherian (\ie if $\Spec Z(\sA)$ is finite), the canonical functor
\[\sA\to \prod_{M\in \Spec Z(\sA)} \sA\sslash M\]
is an equivalence of $\otimes$-categories.\qed
\end{cor}

\begin{cor}\label{c2} Let $F:\sA\to \sB$ be a $\otimes$-functor in $\Ex^\rig$, with $\sB$ connected. Then there is a unique maximal ideal $M$ of $Z(\sA)$ such that $F$ factors through $\sA\sslash M$ (into a faithful $\otimes$-functor). We write $M=M(F)$.
\end{cor}

\begin{proof} Let $\sI=\{A\in \sA\mid F(A)=0\}$: this is a Serre $\otimes$-ideal of $\sA$ such that the induced $\otimes$-functor $\sA\sslash \sI\to \sB$ is faithful. In particular, $Z(\sA\sslash \sI)\inj Z(\sB)$ is a field by \cite[Prop. 2.5 d)]{BVK}. But $Z(\sA\sslash \sI)=Z(\sA)/I(\sI)$ by Theorem \ref{p4}. Let $M=I(\sI)$; Then $\sI=\sI(M)$ by Remark \ref{r2}.
\end{proof}
 
 The following strengthens Corollary \ref{c1} and will be used in the next section.
 
 \begin{prop}\label{p11} Let $X$ be a set of morphisms of $\sA$. Then the set
 \[\{M\in \Spec Z(\sA)\mid f(M)\ne 0\;\forall f\in X\}\]
 is closed. Here $f(M)$ denotes the image of $f$ in $\sA\sslash M$.
 \end{prop}
 
 \begin{proof} Suppose first that $X$ consists of one element $f$. Then $f(M)\ne 0$ $\iff$ $\IM f\notin \sI(M)$ $\iff$ $e(\IM f)\notin M$ $\iff$ $1-e(\IM f)\in M$. The set of those $M$ is closed. The general case follows, since an intersection of closed subsets is closed. 
  \end{proof}

\section{Prime $\otimes$-ideals}\label{s5}

\subsection{Spectral spaces and the constructible topology}\label{s5.3} The constructible topology was developed in \cite[I, \S 7.2]{EGA}. A simpler variant was independently studied by Hochster in \cite{hochster} under the name of ``patch topology''; Hochster also introduced the notion of spectral spaces and spectral maps. Following the Stacks project \url{https://stacks.math.columbia.edu/tag/08YF}, we take Hochster's viewpoint but use ``constructible topology'' in place of ``patch topology''.

\begin{defn}[Hochster] \label{d6} a) A topological space is \emph{spectral} if it is $T_0$ and quasi-compact, the quasi-compact open subsets are closed under finite intersection and form an open basis, and every nonempty irreducible closed subset has a generic point. A continuous map of spectral spaces is \emph{spectral} if inverse images of quasi-compact open sets are quasi-compact. \\
b) Let $X$ be a spectral space. The \emph{constructible topology} $T(X)$ on $X$ is the topology which has the quasi-compact open sets and their complements as an open sub-basis.
\end{defn}

As an example, if $f:A\to B$ is a homomorphism of commutative rings, then $\Spec A$ and $\Spec B$ are spectral spaces and $f^*:\Spec B\to \Spec A$ is a spectral map.

The following is a spectral analogue of Olivier's theorem \cite[Prop. 5]{olivier}:

\begin{thm}[\protect{\cite[\S 9]{hochster}}]\label{t7} a) For any spectral space $X$, $T(X)$ is profinite (\ie compact (Hausdorff) and totally disconnected). Any profinite space is spectral.\\
b) Let $\bS$ be the category of spectral spaces and spectral maps, and let $\P$ be its full subcategory of profinite spaces. Then $T$ defines a right adjoint to the inclusion $\P\inj \bS$.
\end{thm}

\subsection{Zariski $\otimes$-topology} Recall that $\sC\in \Add^\otimes$ is \emph{integral} if $f\otimes g=0$ $\Rightarrow$ $f=0$ or $g=0$.

\begin{defn}\label{d5} Let $\sC\in \Add^\otimes$. A $\otimes$-ideal $\bP\subsetneq \sC$ is \emph{prime} if $\sC/\bP$ is integral. We write $\Spec^\otimes \sC$ for the set of its prime $\otimes$-ideals, and provide it with the  topology whose closed subsets are given by the
\[V(\bI)=\{\bP\mid \bI\subseteq \bP\}\]
for $\bI$ a $\otimes$-ideal of $\sC$.
\end{defn}

Since $V(\bI\otimes \bJ)=V(\bI)\cup V(\bJ)$ and $\bigcap V (\bI_\alpha) = V (\sum \bI_\alpha)$, this indeed defines a topology.

\begin{ex}\label{ex2} For a commutative ring $R$, let $L(R)$ be the $\otimes$-category of finitely generated free $R$-modules. There is a $1-1$-correspondence between $\otimes$-ideals of $L(R)$ and ideals of $R$, obtained by restricting a $\otimes$-ideal to $R=L(R)(\un,\un)$. In particular, $\Spec^\otimes L(R)$ is canonically homeomorphic to $\Spec R$. 
\end{ex}

The main result of this subsection is:

\begin{thm}\label{t8} For any $\sC\in \Add^\otimes$, $\Spec^\otimes\sC$ is a spectral space; for any $F\in \Add^\otimes(\sC, \sD)$, $F^*:\Spec^\otimes \sD\to \Spec^\otimes \sC$ is a spectral map.
\end{thm}

The proof will be given after Corollary \ref{c7}.

Let $\bI$ be a $\otimes$-ideal of $\sC$, and let $f\in \sC(A,B)$. We say that $f$ is \emph{contained in $\bI$} if $f\in \bI(A,B)$. Since an intersection of $\otimes$-ideals is a $\otimes$-ideal, there is a smallest $\otimes$-ideal $\bI$ containing a given set $E$ of morphisms of $\sC$: we say that $\bI$ is \emph{generated by $E$}. We may generalise the definition of $V(\bI)$ to $V(E)$ for any such $E$, and then $V(E)=V(\bI)$ where $\bI$ is the $\otimes$-ideal generated by $E$. 

\begin{lemma}\label{l9} a) Let $f\in Ar(\sC)$. Then the $\otimes$-ideal $(f)$ generated by $f$ is the set of compositions $h\circ (f\otimes 1_C)\circ g$ making sense, where $C\in Ob(\sC)$.\\
b) A $\otimes$-ideal $I$ of $\sC$ is proper if and only if $1_\un\notin I$.\\
c) A morphism $f:A\to B$ is not contained in any maximal $\otimes$-ideal if and only if $1_\un$ may be factored as
\[\un\by{h} A\otimes C\by{f\otimes 1_C} B\otimes C\by{g} \un\]
for some $C \in Ob(\sC)$.\\
d) If $B=\un$ in c), we may choose $C=\un$ and $g=1_\un$, \ie find a section to $f$.
\end{lemma}

\begin{proof} a) This set is obviously closed under left and right external compositions and just as obviously under direct sums; by the usual Mac Lane diagonal/codiagonal trick, this implies that it is closed under ordinary sums. To conclude, it suffices to show that $(f\otimes 1_C)\otimes g\in (f)$ for any $C\in Ob(\sC)$ and any $g\in Ar(\sC)$. But
\[(f\otimes 1_C)\otimes g = (1\otimes g)\circ(f\otimes 1_{C\otimes D})\]
where $D$ is the domain of $g$.

b) ``If'' is obvious. Conversely, if $1_\un\in I$ then so does $1_C=1_\un\otimes 1_C$ for any $C\in \sC$, hence any $f\in Ar(\sC)$.

c) By a), this condition is equivalent to $1_\un\in (f)$, and we conclude by b).

d) Use the commutation $g\circ (f\otimes 1_C)=f\otimes g=f\circ (1_C\otimes g)$.
\end{proof}

\begin{lemma}\label{l10} a) Any maximal $\otimes$-ideal is prime, and any proper $\otimes$-ideal $\bI$ is contained in a maximal $\otimes$-ideal; in particular, $V(\bI)\neq \emptyset$. Any prime $\otimes$-ideal contains a minimal prime $\otimes$-ideal.\\
b) The $D(f):= \Spec^\otimes \sC - V(\{f\})$ for $f\in Ar(\sC)$ form a basis of open sets for the Zariski $\otimes$-topology.\\ 
c) One has $D(\{f_1,\dots,f_n\})=D(f_1\oplus\dots \oplus f_n)$ and $D(f\otimes g)=D(f)\cap D(g)$.\\
d) Let $F:\sC\to \sD$ be a $\otimes$-functor, with $\sC,\sD\in \Add^\otimes$. Then $F$ induces a continuous map $F^*:\Spec^\otimes \sD\to \Spec^\otimes \sC$. If $F$ is full and essentially surjective, $F^*$ is a closed immersion.\\
e) In d), if $F$ is the pseudo-abelian completion of $\sC$, $F^*$ is a homeomorphism.
\end{lemma}

\begin{proof} Only the first statements of a) and d) deserve a proof. For a), we reduce to showing that if $\sC$ has no proper nonzero $\otimes$-ideals, then $\sC$ is integral. Let $f\in Ar(\sC)-\{0\}$: the set of $g$ such that $f\otimes g=0$ is a proper $\otimes$-ideal of $\sC$, hence must be $0$. For d), the point is that $F^{-1}(\bI)$ is a proper $\otimes$-ideal of $\sC$ if $\bI$ is a proper $\otimes$-ideal of $\sD$, which follows from Lemma \ref{l9} b). Finally, e) is clear since any $\otimes$-ideal $\bI$ of $\sC$ extends uniquely to a $\otimes$-ideal of $\sD$, prime if $\bI$ is prime.
\end{proof}

The following is a $\otimes$-analogue of a well-known result of commutative algebra:

\begin{prop}\label{p13} For any $\otimes$-ideal $\bI$ of $\sC\in \Add^\otimes$, the ideal $\sqrt[\otimes]{\bI}$ of \cite[Def. 7.4.1]{AK2} is the intersection of the prime $\otimes$-ideals of $\sC$ containing $\bI$.
\end{prop}

\begin{proof} We reduce to $\bI=0$. Let $f\notin \sqrt[\otimes]{0}$. We must find $\bP\in\Spec^\otimes \sC$ such that $f\notin \bP$. We use the idea of \cite[\S 3]{os} (fractional closure), that we apply to the not necessarily regular morphism $f:A\to A'$. Thus, let $\sC[1/f]$ be the category with the same objects as $\sC$ and morphisms given by
\[ \sC[1/f](C,C')=\colim_n \sC_{f^{\otimes n}}(C,C')\]
for $C,C'\in \sC$, where $\sC_{f^{\otimes n}}(C,C')$ is the subgroup of $\sC(A^{\otimes n}\otimes C,{A'}^{\otimes n}\otimes C')$ defined in the commutation of diagrams (3.1) in \cite{os}. This is a rigid $\otimes$-category by the same arguments as in \loccit Moreover, $\sC[1/f]\ne 0$: to see this it suffices to see that $1_\un$, the identity endomorphism of $\un\in \sC$, does not map to $0$ in $\sC[1/f](\un,\un)$. But this is clear, since the image of $1_\un$ in $\sC_{f^{\otimes n}}(\un,\un)$ is $f^{\otimes n}$. Now, the inverse image of any prime $\otimes$-ideal of $\sC[1/f]$ by the $\otimes$-functor $\sC\to \sC[1/f]$ is the desired prime $\otimes$-ideal $\bP$.
\end{proof}

This proof is inspired by the one in commutative algebra using localisation. As O'Sullivan pointed out, there is a more elementary proof in the style of other proofs from commutative algebra as in \cite[proof of Prop. 1.8]{am}.

\begin{cor}\label{c7} For $\sC\in \Add^\otimes$,\\
a) An open subset of $\Spec^\otimes \sC$ is quasi-compact if and only if it is of the form $D(f)$. In particular, $\Spec^\otimes \sC$ is quasi-compact and its quasi-compact open subsets are closed under finite intersections.\\
b) For $Y\subseteq \Spec^\otimes \sC$, let $\bI(Y)=\bigcap_{\bP\in Y} \bP$. Then $\bI(Y)=\sqrt[\otimes]{\bI(Y)}$. Moreover, $Y$ is irreducible if and only if $\bI(Y)$ is prime. The map $\bP\mapsto V(\bP)$ is a bijection of $\Spec^\otimes \sC$ onto the set of irreducible closed subsets of  $\Spec^\otimes \sC$.
\end{cor}

\begin{proof} As usual: see Lemma \ref{l10} b) and c),  \cite[Ch. II, \S 4, Prop. 11, 12 and 14]{bbki} and \cite[I, (1.1.4)]{EGA}. \end{proof}

Theorem \ref{t8} follows from Corollary \ref{c7}, noting also that $(F^*)^{-1}(V(f))\allowbreak=V(F(f))$ for any $f\in Ar(\sC)$.

\begin{defn}\label{d2} Let $\sC\in \Add^\otimes$. A morphism $f\in \sC$ is \emph{quasi-invertible} if it verifies the equivalent conditions of Lemma \ref{l9} c).
\end{defn}

\begin{lemma}\label{l15} If $\sC\in \Add^\rig$,\\
a)  $f:A\to B$ is quasi-invertible if and only if its right adjoint $\tilde f:A\otimes B^\vee\to \un$ is quasi-invertible.\\
b) Any quasi-invertible morphism  is strongly regular in the sense of \cite[\S 3]{os}.
\end{lemma}

\begin{proof} a) follows once again from (the dual of) \cite[Lemma 6.1.5]{AK2}. b) by a) and its analogue for strongly regular morphisms \cite[bot. p. 9]{os}, we may assume  that $B=\un$ in a) and therefore, by Lemma  \ref{l9} d), that $f$ has a section $g$. First, $f$ is regular: if $f\otimes h=0$, then $0=(f\otimes h)\circ (g\otimes 1)=1\otimes h=h$. As observed in \cite[top p. 10]{os}, to prove that $f$ is strongly regular, it suffices to prove that $\sC(\un,D)\by{f\otimes-}\sC_f(\un,D)$ is surjective for any $D\in \sC$. Let $h\in  \sC_f(\un,D)$, \ie $h:A\to D$ makes the diagram (3.1) of \cite{os} commute. In this special case, this means that $h$ verifies the identity
\[f\otimes h=(f\otimes h)\circ c_{A,A}\]
where $c$ is the symmetry of $\sC$. Composing again with $g\otimes 1$ on the right, this gives
\begin{multline*}
h=1_\un\otimes h=(f\otimes h)\circ (g\otimes 1) =(f\otimes h)\circ c_{A,A}\circ  (g\otimes 1)\\
 = (f\otimes h)\circ(1\otimes g)\circ c_{A,\un}=f\otimes (h\circ g). 
\end{multline*}
\end{proof}

(The converse of Lemma \ref{l15} b) is false: take $\sC=\Rep_K(\G_a)$ for a field $K$, and $f:V\surj \un$ where $V$ is the standard $2$-dimensional representation of $\G_a$. Then $f$ does not have a section, but is strongly regular because it is regular and $\sC$ is fractionally closed by \cite[Lemma 3.1]{os}.)

\subsection{Relationship between spectra and $\otimes$-spectra} Let $\sC\in \Add^\otimes$. By Example \ref{ex2} and Corollary \ref{c7} d), the obvious $\otimes$-functor $L(Z(\sC))\to \sC$ induces a spectral map
\begin{equation}\label{eq2b}
 \pi:\Spec^\otimes \sC\to \Spec Z(\sC).
\end{equation}

 If $\bI$ is a $\otimes$-ideal of $\sC$, write more generally $\pi(\bI)=\bI(\un,\un)$ for the corresponding ideal of $Z(\sC)$. Conversely, if $I$ is an ideal of $Z(\sC)$, we get a $\otimes$-ideal $\bI(I)$ of $\sC$ by the formula
\[\bI(I)(C,D) = I\cdot\sC(C,D)\]
for the action of $Z(\sC)$ on $\sC$. 

\begin{lemma}\label{l8} We have $\pi(\bI(I))=I$ and $\bI(\pi(\bI))\subseteq \bI$.
\end{lemma}

\begin{proof} The first point is obvious. For the inclusion, let $C,D\in \sC$. We may write any $f\in \bI(\pi(\bI))(C,D)$ as a linear combination $\sum_\alpha z_\alpha\otimes  f_\alpha$ with $z_\alpha\in \bI(\un,\un)$ and $f_\alpha\in \sC(C,D)$. Thus $f\in \bI(C,D)$.
\end{proof}



\subsection{Two continuous sections} The map $I\mapsto \bI(I)$ does not send prime ideals to prime $\otimes$-ideals in general, so there is no map in the opposite direction to  \eqref{eq2b} \emph{a priori}. This situation improves in the rigid case:

If $\sC\in \Add^\rig$ and $I\subset Z(\sC)$ is an ideal, we get a $\otimes$-ideal of $\sC$ by the formula
\[\tr^*(I)(C,D)=\{f:C\to D\mid \tr(gf)\in I\ \forall\ g:D\to C\}.\]

\begin{prop}\label{ex1} a) We have $\bI(I)\subseteq \tr^*(I)$ and $I(\tr^*(I))=I$.\\
b) If $P$ is prime, so is $\tr^*(P)$. This defines a continuous and closed section $\sigma_\tr$ of  \eqref{eq2b}.\\ 
c) If $P$ is maximal, $\sigma_\tr(P)$ is maximal.\\
d) Any $\bP\in \Spec^\otimes \sC$ is contained in $\sigma_\tr\pi(\bP)$.\\
e) $\sigma_\tr$ is spectral if and only if the following condition holds:
\begin{quote}
For any $D\in \sC$ and any $f\in \sC(\un,D)$, the ideal 
\[\{gf\mid g\in \sC(D,\un)\}\] of $Z(\sC)$ 
is finitely generated.
\end{quote}
(This is automatic is $Z(\sC)$ is Noetherian.)
\end{prop}

\begin{proof} a) is obvious. 
b) For the first point, it suffices as usual to show that  $f\otimes g\notin \tr^*(P)$ if $f\notin \tr^*(P)$ and $g\notin \tr^*(P)$ for $f$ and $g$ with domain $\un$. But, by hypothesis, there exist $f'$, $g'$ such that $f'\circ f\notin \tr^*(P)(\un,\un)=P$ and $g'\circ g\notin P$. Then $(f'\otimes g')\circ (f\otimes g)=(f'\circ f)\otimes (g'\otimes g)\notin P$. 

That $\sigma_\tr$ is a section of $\pi$ follows from a). Let $\bI$ be a $\otimes$-ideal of $\sC$. We claim that
\[(\tr^*)^{-1}(V(\bI))=V(\pi(\bI))\]
which will prove that $\sigma_\tr$ is continuous. Indeed, let $P\in\Spec Z(\sC)$ be such that $I\subseteq \tr^*(P)$. The $\sI(I)=I(\un,\un)\subseteq \tr^*(P)(\un,\un)=P$. Conversely if this holds, then, for any $D\in \sC$, any $f\in I(\un,D)$ and any $g\in \sC(D,\un)$, $\tr(gf)=gf\in I(\un,\un)\subseteq P$.

To see that $\sigma_\tr$ is closed, we observe that the obvious inclusion $\tr^* V(I)\subseteq V(\tr^*(I))$ is an equality for all $I$, thanks to a).

c) We reduce to $P=0$ by replacing $\sC$ with $Z(\sC)/P\allowbreak \otimes_{Z(\sC)} \sC$. Then $Z(\sC)$ is  a field and $\tr^*(0)=\sN_\sC$, which is the unique maximal proper $\otimes$-ideal of $\sC$ by \cite[Prop. 7.1.4 b)]{AK2}. (One could say that $\sC$ is \emph{$\otimes$-local} in this case.) 

d) More generally, $\bI\subseteq \tr^*\pi(\bI)$ for any $\otimes$-ideal $\bI$ (direct check).

Finally, the ideal displayed in e) is none else than $(\tr^*)^{-1}(V(f))$, and $V(g)=V(\tilde g)$ for any $g:C\to D$ where $\tilde g$ is the right adjoint of $g$.
\end{proof}

Let now $\sA\in \Ex^\rig$. If $I$ is an ideal of $Z(\sA)$, define
\[I_\oplus:=\Ker_m(\sA\to \sA\sslash I)\]
\cf Section \ref{s2a} and Notation \ref{n1}. By Lemma \ref{l7} and Remark \ref{r2}, this is the $\otimes$-ideal
\[\{f\mid e(f)\in I\}\]
(see Definition \ref{d4}).

\begin{prop} \label{p17} a) We have $ I_\oplus\subseteq \bI(I)$ and $I(I_\oplus)=I$. \\
b) For $P\in \Spec Z(\sA)$, $P_\oplus$ is prime. This defines another section of \eqref{eq2b}:
\begin{equation}\label{eq11}
\sigma:\Spec Z(\sA)\to \Spec^\otimes \sA
\end{equation}
such that $\sigma(P)\subseteq \sigma_\tr(P)$ for all $P$ (\cf Proposition \ref{ex1}). It is spectral.\\
c) Any $\bP\in \Spec^\otimes \sA$ contains $\sigma\pi(\bP)$.\\
d) For any $I$, the obvious $\otimes$-functor
\[\sA/I_\oplus\to \sA\sslash I\]
is an equivalence of categories.
\end{prop}

\begin{proof} In a), the first point   follows from Lemma \ref{l2} c) and the second one is obvious.

b) For the first claim, we use the fact that $Z(\sA\sslash \sI)$ is a field (Theorem \ref{p4}), hence that $\sA\sslash \sI$ is integral (\cite[Prop. 2.5 a)]{BVK} or Proposition \ref{p2}). This defines \eqref{eq11}.

Let $X\subset Ar(\sA)$. By Lemma \ref{l7}, $P\in \Spec Z(\sA)$ is in the inverse image of $V(X)$ under  \eqref{eq11} if and only if $\IM(f)\in \sI(P)$ for all $f\in X$. By Remark \ref{r2}, this condition amounts to $e(\IM(f))\in P$ for all $f\in X$, \ie to $P\in V(\{e(\IM(f))\mid f\in X\})$. This shows the spectraliity of \eqref{eq11}, and the fact that it is a section of \eqref{eq2b} follows from a). 

c) More generally we have $(\pi(\bI))_\oplus\subseteq \bI$ for any $\bI$, as follows from a) and Lemma \ref{l8}. 

In d), the functor is faithful and (essentially) surjective by definition, and is full by Theorem \ref{p4}.
\end{proof}

\begin{rk}
Propositions \ref{ex1} and \ref{p17} give the following nice picture of $\Spec^\otimes \sC$ for $\sC\in \Add^\rig$: it is fibred over $\Spec Z(\sC)$ by the continuous map $\pi$ of \eqref{eq2b} and, for any $P\in \Spec Z(\sC)$, $\pi^{-1}(P)$ is homeomorphic to $\Spec^\otimes(\sC/P)$. There is a canonical ``maximal'' section, as well as a ``minimal''  $0$-section \eqref{eq11} when $\sC\in \Ex^\rig$.
\end{rk}

\section{Applications to universal rigid abelian $\otimes$-categories}

Let $\sC\in \Add^\rig$. We write $\lambda_\sC:\sC\to T(\sC)$ for the canonical $\otimes$-functor to the universal abelian $\otimes$-category $T(\sC)$ from \cite[Th. 5.1]{BVK}.

\subsection{(Local) abelian $\otimes$-envelopes}
The following corollary is a complement to \cite[Prop. 6.2]{BVK}:

\begin{cor}\label{c3} Suppose that $\sC$ admits a faithful $\otimes$-functor to a connected rigid abelian $\otimes$-category $\sB$. Let $F:T(\sC)\to \sB$ be the induced exact $\otimes$-functor. If $\sC$ admits an abelian $\otimes$-envelope $E(\sC)$ in the sense of \cite[Def. 6.1]{BVK}, then $E(\sC)=T(\sC)\sslash M(F)$ (see Corollary \ref{c2}).\qed
\end{cor}

What happens when $\sC$ does not admit an abelian $\otimes$-envelope? The answer can be seen as a more precise version of \cite[Th. 5.2.2]{coul3}:

\begin{cor}\label{c5} The faithful $\otimes$-functors to connected rigid abelian $\otimes$-categories are classified by those maximal ideals $M\in \Spec Z(T(\sC))$ such that the composition $\sC\by{\lambda_\sC} T(\sC)\to T(\sC)\sslash M$ is faithful. This set is closed, hence compact (and may be empty). We call these functors the \emph{local abelian $\otimes$-envelopes} of $\sC$.
\end{cor}

(Said otherwise: the compact set of corollary \ref{c5} is in $1-1$ correspondence with the set $\sH\sK(\sC)$ of \cite[Th. 5.2.2 (1)]{coul3}.)

\begin{proof} The only thing to justify is the closedness, which follows from Proposition \ref{p11}.
\end{proof}

\begin{thm} \label{p14} Let $\bI$ be a $\otimes$-ideal of $\sC$. Then the canonical functor $F:T(\sC)\to T(\sC/\bI)$ is a localisation, and its (object) kernel $\sI$ is generated by the $\Supp \lambda_\sC(f)$ for $f\in \bI$, hence the ideal $I(\sI)$ of Remark \ref{r2} equals $\pi(\bI)=\bI(\un,\un)$. The induced functor 
 \[Z(T(\sC))/\pi(\bI)\otimes_{Z(T(\sC))} T(\sC)\to T(\sC/\bI)\]
is an equivalence of $\otimes$-categories.
\end{thm}

\begin{proof} $F$ factors as a composition of exact $\otimes$-functors
\[T(\sC)\to T(\sC)\sslash \sI\by{\bar F} T(\sC/\bI)\]
where $\bar F$ is conservative, hence faithful.  Thus a $\otimes$-functor from $\sC$ to an abelian $\otimes$-category $\sA$ factors through $\sC/\bI$ if and only if its $\otimes$-exact extension to $T(\sC)$ factors through $T(\sC)\sslash \sI$. Therefore $T(\sC)\sslash \sI$ has the same $2$-universal property as $T(\sC/\bI)$, and $\bar F$ is an equivalence.

If $f\in I$, then clearly $\Supp \lambda_\sC(f)\in \sI$. Conversely, let $\sI'$ be the Serre $\otimes$-ideal generated by the $\Supp \lambda_\sC(f)$ for $f\in I$. Then the composition
\[\sC\by{\lambda_\sC} T(\sC)\to T(\sC)\sslash \sI'\]
factors through $\sC/I$.  Hence the localisation functor $T(\sC)\to T(\sC)\sslash \sI'$ factors through $T(\sC)\sslash \sI$, which implies $\sI'=\sI$. The last claim follows from Theorem \ref{p4}.
\end{proof}

\begin{ex}\label{ex3} Let $k$ be a field; for $\sim$ an adequate equivalence relation on algebraic cycles, write $\sM_\sim(k)$ for the category of motives modulo $\sim$ with rational coefficients, and let $T_\sim(k)=T(\sM_\sim(k))$. Then, if $\sim\ge \sim'$, the natural functor $T_\sim(k)\to T_{\sim'}(k)$ is a (full) localisation, and an equivalence of categories if $T_\sim(k)$ is connected.
\end{ex}


\subsection{Commutation with colimits}

\begin{prop}\label{p23} Let $(\sC_i)_{i\in I}$ be a $2$-direct system in $\Add^\rig$, and suppose that the $2$-colimit $\sC=\colim_i \sC_i$ exists. Then so does $\colim_i T(\sC_i)$ in $\Ex^\rig$, and the natural functor
\[\colim\nolimits_i T(\sC_i)\to T(\sC)\]
is an equivalence of $\otimes$-categories. In particular, $\colim_i Z(T(\sC_i))\iso Z(T(\sC))$.
\end{prop}

\begin{proof} This follows from the $2$-universal property of $T$ (``a left adjoint commutes with arbitrary colimits'').
\end{proof}

\subsection{Prime $\otimes$-ideals and abelian $\otimes$-envelopes}
Let $\sC\in \Add^\rig$. 
The map \eqref{eq11} for $\sA=T(\sC)$ and the contravariance of $\Spec^\otimes$ define a continuous map
\begin{equation}\label{eq2a}
\Spec Z(T(\sC))\to \Spec^\otimes \sC.
\end{equation}



\begin{prop}\label{p15} The map \eqref{eq2a} is spectral. 
The fibre of a prime $\otimes$-ideal $\bP\in \Spec^\otimes \sC$ under \eqref{eq2a}  is in $1-1$ correspondence with the set of local abelian $\otimes$-envelopes of $\sC/\bP$. 
\end{prop}

(Considering the constructible topology on $\Spec^\otimes \sC$, we recover the closedness statement of Corollary \ref{c5} in a more conceptual way.)

\begin{proof} The first claim follows from Corolary \ref{c7} d) and Proposition \ref{p17} b). 
For the next one, 
let us first assume $\bP=0$. Then the statement follows  from Corollary \ref{c5}. In general, let $F:\sC\to \sC/\bP$ be the projection, and let $M\in \Spec Z(T(\sC))$ be in the fibre of \eqref{eq2a}. 
We then have a naturally commutative diagram
\[\xymatrix{
\sC\ar[r]\ar[d]^F& T(\sC)\ar[r]\ar[d]^{T(F)}& T(\sC)\sslash M\\
\sC/\bP\ar[r]& T(\sC/\bP)\ar[ru]_{G}
}\]
where $G$ is a localisation as a consequence of Theorem \ref{p14}. Thus $M$ defines a local abelian $\otimes$-envelope of $\sC/\bP$. Conversely, any local abelian $\otimes$-envelope of $\sC/\bP$ gives rise to such a diagram, thus comes from an ideal in the fibre of \eqref{eq2a}.
\end{proof}

\begin{cor}\label{c10} Suppose that $Z(\sC)$ is a field and that $S(\sC):=(\sC/\sN_\sC)^\natural$ is split \cite[Prop. 7.3]{BVK}. If  $\sN_\sC$ is finitely generated, the functor $\bar \pi:T(\sC)\to T(S(\sC))=S(\sC)$ of \loccit (7.1) yields a canonical decomposition
\[T(\sC)\simeq \Ker \bar \pi \times S(\sC).\]
\end{cor}

\begin{proof} By Proposition \ref{p15}, the fibre of $\sN_\sC$ under \eqref{eq2a} has one element; if  $\sN_\sC$ is finitely generated, the open subset $\Spec^\otimes \sC -\{\sN_\sC\}$ of $\Spec^\otimes \sC$ is quasi-compact, hence clopen in $\Spec Z(T(\sC))$. Therefore the corresponding maximal ideal is generated by an idempotent, which induces the desired decomposition  by \cite[Cor. 5.2]{BVK}.
\end{proof}

\begin{rk} Of course there are many cases when the finite generation hypothesis fails, \eg from Example \ref{ex2}.
\end{rk}

\subsection{The case of an abelian category} Let $\sA\in \Ex^\rig$. The functor $\lambda_\sA:\sA\to T(\sA)$ then has a canonical exact $\otimes$-retraction $\rho_\sA:T(\sA)\to \sA$ \cite[Cor. 5.2]{BVK}. By Theorem \ref{t1}, $\rho_\sA$ factors through an exact faithful $\otimes$-functor
\begin{equation}\label{eq10}
\bar \rho_\sA:T(\sA)\sslash I\to \sA
\end{equation}
for a unique ideal $I$ of $Z(T(\sA))$ (see Notation \ref{n1}).

\begin{prop} The functor $\bar \rho_\sA$ is an equivalence of categories, and $I$ is maximal if and only if $\sA$ is connected.
\end{prop}

\begin{proof} Write $\bar \lambda_\sA$ for the composite $\sA\by{\lambda_\sA}T(\sA)\to T(\sA)\sslash I$, so that $\bar \rho_\sA\bar \lambda_\sA=\Id_\sA$. This already shows that $\bar \rho_\sA$ is essentially surjective; to show its fullness, it suffices then to see that $\bar \lambda_\sA$ is essentially surjective. We proceed in two steps:

1) $\bar \lambda_\sA$ is exact: indeed, its composition with the faithful exact functor $\bar \rho_\sA$ is exact.

2)  $\bar \lambda_\sA$ is essentially surjective. Since any object of $T(\sA)$, hence of $T(\sA)\sslash I$, is isomorphic to a subquotient of an object coming from $\sA$ \cite[Prop. 4.4]{BVK}, it suffices to show that the essential image of $\bar \lambda_\sA$ is stable by subobjects. Let $A\in \sA$, and let $M\subseteq \bar\lambda_\sA(A)$. By 1), $M'=\bar \lambda_\sA\bar \rho_\sA(M)$ is another subobject of $\bar\lambda_\sA(A)$. Let $M''=M+M'\subseteq \bar\lambda_\sA(A)$; then  $\bar\rho_\sA(M)=\bar \rho_\sA(M'')=\bar \rho_\sA(M')$, hence $M=M''=M'$ by using again the faithful exactness of $\bar\rho_\sA$. 

The last equivalence follows from Theorem \ref{p4} and Corollary \ref{c2}. (In this connected case, the proposition also follows directly from Corollary \ref{c3}.)
\end{proof}

\section{Schur finiteness} 

\subsection{Basics} Recall the following definition:

\begin{defn} Let $\sC\in \Add^\otimes$ be $\Q$-linear. We say that an object $C$ of $\sC$ is \emph{Schur-finite} if there exists a Schur functor $S$ \cite{dtens} such that $S(C)=0$, and that $\sC$ is \emph{Schur-finite} if all its objects are Schur finite.
\end{defn}

Schur-finiteness has the following stability properties:

\begin{prop}\label{p5} Let $\sC\in \Add^\otimes$ be $\Q$-linear.\\
o) $\un$ is Schur-finite.\\
a) Let $C, C'\in \sC$ If $C$ and $C'$ are Schur-finite, then $C\oplus C'$ and $C\otimes C'$ are Schur-finite. So is any direct summand of $C$. If $C$ is Schur-finite and dualisable, its dual is Schur-finite.\\
b) If the tensor structure of $\sC$ respects monomorphisms (\resp epimorphisms), any subobject (\resp quotient) of a Schur-finite object is Schur-finite.\\
c) Suppose that $\sC$ is abelian and that its tensor structure is right exact. Let $C'\to C\to C''\to 0$ be an exact sequence in $\sC$. Then for any Schur functor $S$, $S(C'\oplus C'')=0$ implies $S(C)=0$.\\
d) If $\sC$ is abelian and its tensor structure is exact, the full subcategory of Schur-finite objects is a Serre subcategory of $\sC$, stable under (internal) tensor product.
\end{prop}

\begin{proof} o) $\Lambda^2(\un)=0$. a) is \cite[Cor. 1.13 and 1.18]{dtens}. b) is obvious. c) follows from the proof of \cite[Prop. 2.17]{exandfaith}. d) follows from a), b) and c) (see also \cite[Prop. 1.19]{dtens}).
\end{proof}

\begin{prop} In the situation of \cite[Prop. 3.2]{BVK}, $\sA$ is Schur-finite if and only if $\sI$ and $\sA\sslash \sI$ are.
\end{prop}

\begin{proof} ``Only if'' is trivial. For ``if'', let $A\in \sA$, and suppose that its image in $\sA\sslash \sI$ is killed by a Schur functor $S$. Then $S(A)\in \sI$, hence there exists another Schur functor $S'$ such that $S'(S(A))=0$. By \cite[Ex. 6.17]{FH}, this is a nontrivial direct sum of objects of the form $S''(A)$, so $A$ is Schur-finite.
\end{proof}

We also recall the following theorems of Deligne and O'Sullivan:

\begin{thm}\label{t4} a) (Deligne) Let $\sA\in \Ex^\rig$ be $\Q$-linear, Schur-finite and connected. Then there exists an extension $L/K$ and an exact (hence faithful) $\otimes$-functor $\omega:\sA\to \Vec_L^\pm$, where the latter category is that of $\Z/2$-graded finite-dimensional $L$-vector spaces (with the commutativity constraint given by the Koszul rule).\\
b) (O'Sullivan) Let $\sC\in \Ex^\rig$ be integral and Schur-finite. Then $\sC$ admits a faithful $\otimes$-functor to a category $\sA\in \Ex^\rig$ as in a), and even an initial one.
\end{thm}

\begin{proof} a) follows easily from \cite[Prop. 2.1]{dtens} (for details, see \cite[Ex. 2.9 b)]{exandfaith}). b) is \cite[Th. 10.10]{os}; see also Theorem \ref{t3} below.
\end{proof}

\begin{defn} For $\sA\in \Add^\otimes$, a \emph{weak fibre functor} is a $\otimes$-functor $\sA\to \Vec_L^\pm$.
\end{defn}

\begin{prop}\label{p8} A $\Q$-linear category $\sA\in \Ex^\rig$ is Schur-finite if and only if it admits a conservative system of  weak fibre functors.
\end{prop}

\begin{proof} Combine Theorem \ref{t4} a) and Corollary \ref{c1}.
\end{proof}

\subsection{Classification}

\begin{prop}\label{p7} Let $\sC\in \Add^\rig$. Then $\sC$ Schur-finite $\Rightarrow$ $T(\sC)$ Schur-finite; the converse is true if $\Ker(\sC\to T(\sC))$ is a nilideal of $\sC$.
\end{prop}

\begin{proof} As observed in \cite[Lemma 4.1]{BVK}, any object of $\Ab(\sC)$ is isomorphic to a subquotient of an object of the form $\iota_\sC(C)$ for $C\in \sC$. This carries over to its localisation $T(\sC)$. The first statement then follows from Proposition \ref{p5} b) since the tensor structure of $T(\sC)$ is exact. Conversely, suppose $T(\sC)$ Schur-finite, and let $C\in \sC$. By hypothesis, there is a Schur functor $S$ such that $1_{S(C)}\in \Ker(\sC\to T(\sC))$. If the latter is a nilideal, we must have $1_{S(C)}=0$, \ie $C=0$.
\end{proof}

\begin{thm}\label{t3} If $T(\sC)$ is Schur-finite, it is $2$-universal for $\otimes$-functors from $\sC$ to Schur-finite rigid abelian $\otimes$-categories. In particular, O'Sull\-iv\-an's hulls of  \cite[Lemma 10.7 and Th. 10.10]{os} mentioned in Theorem \ref{t4} b) are abelian $\otimes$-envelopes in the sense of \cite[\S 6.1]{BVK}.
\end{thm}

\begin{proof} This follows from \cite[Th. 5.1]{BVK} and Proposition \ref{p5} d).
\end{proof}

\begin{cor} Let $\sC\in \Add^\rig$ be $\Q$-linear and Schur-finite. Then a morphism $f\in \sC$ maps to $0$ in $T(\sC)$ if and only if it maps to $0$ via any weak fibre functor.
\end{cor}

\begin{proof} ``Only if'' is obvious and ``if'' follows from Proposition \ref{p8}. 
\end{proof}

\begin{thm}\label{c4} For $\sC\in \Add^\rig$ Schur-finite, the map \eqref{eq2a} is a homeomorphism for the constructible topology on $\Spec^\otimes \sC$. Moreover, $\Ker(\lambda_\sC:\sC\to T(\sC))=\sqrt[\otimes]{0}$.\\ 
In particular,  $Z(T(\sC))$ is a field if and only if  $\sC$ has a unique prime $\otimes$-ideal. If $Z(\sC)$ is a field, this is equivalent to saying that $\sN_\sC$ is the only prime ideal of $\sC$.
\end{thm}

\begin{proof} The bijectivity of  \eqref{eq2a} follows from Proposition \ref{p15} and Theorem \ref{t3}; Proposition \ref{p15} also implies that it is a homeomorphism. The second (\resp third) claim then follows from \cite[Th. 5.1]{BVK} and Proposition \ref{p13} (\resp from Proposition \ref{ex1}). 
\end{proof}




As an application, we get the following partial refinement of  \cite[Prop. 8.5]{BVK}:

\begin{cor}\label{c9} Let $k$ be a field and, for an adequate equivalence relation $\sim$ on algebraic cycles, let $\sM_\sim^\ab(k)$ be the thick subcategory of $\sM_\sim(k)$ generated by Artin motives and motives of abelian varieties. Then\\
a) the functor 
\[\sM_\tnil^\ab(k)\to T(\sM_\tnil^\ab(k))\]
is faithful.\\
b) $\sM_\tnil^\ab(k)\to \sM_\num^\ab(k)$ is an equivalence of categories if and only if $T(\sM_\tnil^\ab(k))$ is connected.
\end{cor}

\begin{proof} a) Indeed, $\sM_\tnil^\ab(k)$ is a Kimura category, hence Schur-finite, and we apply Theorem \ref{c4}. b) then follows from a) in the same way as in \cite[Prop. 8.5]{BVK}.
\end{proof}

\begin{rk} In contrast to Theorem \ref{c4}, if $Z(\sC)$ is a field, the map $\Spec Z(T(\sC))\to \Spec^\otimes \sC$ obtained by composing with the section $\sigma_\tr$ of Proposition \ref{ex1} b) has image the maximal $\otimes$-ideal $\sN_\sC$ of $\sC$: this follows from the obvious naturality of this section.
\end{rk}

\subsection{Schur ideals} 

\begin{defn}\label{d7} Let $\sC\in \Add^\otimes$, and let $X$ be a collection of objects of $\sC$.\\ 
a) A $\otimes$-ideal $\bI$ of $\sC$ is \emph{$X$-Schur} if every $C\in X$ becomes Schur-finite in $\sC/\bI$.\\
b) The \emph{$X$-Schur locus} of $\Spec^\otimes \sC$ is
\[\bS(\sC,X) = \{\bP\in \Spec^\otimes \sC\mid \bP \text{ is $X$-Schur}\}.\]
If $X=Ob(\sC)$, we say Schur for $X$-Schur and write $\bS(\sC)$ for $\bS(\sC,X)$. If $X=\{C\}$, we say $C$-Schur for $X$-Schur and write $\bS(\sC,C)$ for $\bS(\sC,X)$.
\end{defn}

Let $\sC\in \Add^\rig$ be integral and Schur-finite. By Theorem \ref{t4}, there exists a faithful $\otimes$-functor $\omega:\sC/\bP\to \Vec_K^\pm$ for some extension $K$ of $\Q$; in other words, $\sC$ is proto-tannakian in the sense of \cite[Def. 2.8]{exandfaith}. By \loccit, Lemma 2.13, the super-dimension of $\omega(C)$ for $C\in \sC$ does not depend on the choice of $\omega$: we call it the \emph{super-dimension of $C$} and write it $\dim^\pm(C)$. 

More generally, suppose only that $\sC$ is integral. By Proposition \ref{p5} a), the full subcategory $\sC'$ of $\sC$ formed of Schur-finite objects still belongs to $\Add^\rig$, and it evidently integral and Schur-finite. Applying the above to $\sC'$, we may define the super-dimension of any Schur-finite object of $\sC$. If $C\in \sC$ is not Schur-finite, we set $\dim^\pm(C)=(\infty|\infty)$.

\begin{defn}\label{d8} Let $\sC\in \Add^\rig$, $C\in \sC$ and $\bP\in \Spec^\otimes \sC$. Letting $\pi_\bP:\sC\to \sC/\bP$ be the projection functor, we let
\[\dim^\pm_\bP(C)=\dim^\pm(\pi_\bP(C)).\]
\end{defn}

\begin{lemma} Let $C\in \sC$.\\
a) For $\bP\in \Spec^\otimes \sC$, write $\chi_\bP(C)=\Tr(1_{\pi_\bP(C)})\in Z(\sC/\bP)$. Then $\chi_\bP(C)$ is independent of $\bP$  if $\Spec^\otimes \sC$ is connected; if $\dim^\pm_\bP(C)=(p|q)$ with $(p|q)<(\infty|\infty)$, then $\chi_\bP(C)=p-q$.\\
b) Suppose $\Spec^\otimes \sC$ connected. If $\chi_\bP(C)\notin\Z$ for some (hence all) $\bP$, then $\bS(\sC,C)= \emptyset$.\\
c) If $\bP\subseteq \bQ$, $\dim^\pm_\bP(C)\ge \dim^\pm_\bQ(C)$. In particular, if $Z(\sC)$ is a field, then $\bS(\sC,C)\neq \emptyset$ $\iff$ $\dim^\pm_\sN(C)<(\infty,\infty)$.
\end{lemma}

\begin{proof} a) The first fact is obvious since the trace commutes with $\otimes$-functors. The second one follows from 
\cite[Lemma 2.13 (5)]{exandfaith}. b) follows from a). c) is clear.
\end{proof}

\begin{rk}\label{r4} It is quite possible that $\bS(\sC) =\emptyset$ even if $\chi(C)\in \Z$ for all $C\in \sC$: this happens \eg if $\sC\in \Add^\rig$ and $T(\sC)=0$, by Theorem \ref{t4} b).
\end{rk}

\begin{prop}\label{p25} a) If $\bI$ and $\bJ$ are $X$-Schur, so is $\bI\otimes \bJ$, hence also $\bI\cap \bJ$.\\
b) Suppose that $\bI\subseteq \bJ$. If $\bI$ is $X$-Schur, so is $\bJ$; the converse is true if $\bJ/\bI$ is nil in $\sC/\bI$.\\
c) If $\sC\in \Add^\rig$, $\bI$ is $X$-Schur if and only if $\sqrt[\otimes]{\bI}$ is.\\
d) $\bI$ is $X$-Schur if and only if $\bI^*$ is (see \S \ref{s2a}).
\end{prop}

\begin{proof} a) follows from \cite[Prop. 1.6]{dtens}. The first part of b) is clear; its second part  is seen as in the proof of Proposition \ref{p7}. c) follows from b) and \cite[Lemma 7.4.2 ii)]{AK2}. In d), $\bI$ is $X$-Schur $\iff$ for all $C\in X$ there exists a Schur functor $S$ such that $1_{S(C)}\in \bI$ $\iff$ $\bI^*$ is $X$-Schur. 
\end{proof}

Let $\Add_s^\otimes$ be the $1$-full, $2$-full subcategory of $\Add^\otimes$ formed of Schur-finite categories, and similarly for $\Add_s^\rig$, $\Ex_s^\otimes$ and $\Ex_s^\rig$. By Proposition \ref{p5} a) and d), the inclusion functors $\Add_s^\otimes\inj \Add^\otimes$, etc. all have a right $2$-adjoint. Similarly:

\begin{cor} The inclusions $\Add_s^\otimes\inj \Add^\otimes$ and $\Add_s^\rig\inj \Add^\rig$ have pro-left $2$-adjoints.
\end{cor}

\begin{proof} This follows from Proposition \ref{p25} a).
\end{proof}

Let $\bI$ be $X$-Schur. Then $V(\bI)\cap \Spec^\otimes \sC\subseteq \bS(\sC,X)$. By Proposition \ref{p25} a) and b), this defines a topology on $\bS(\sC,X)$, coarser than the induced topology from $\Spec^\otimes \sC$. It would be interesting to better understand it in general: for example there may not be any finitely generated $X$-Schur ideal, and in this case $\bS(\sC,X)$ is (probably) not spectral for this topology.


\section{The free rigid $\otimes$-category on one generator} 

\subsection{Introducing the player }Consider the category $\sL\in \Add^\rig$ with distinguished object $L$ constructed in \cite[(1.26)]{dm}: for any $\sC\in \Add^\rig$, the functor $\Add^\rig(\sL,\sC)\to \sC$ sending $F$ to $F(L)$ is an equivalence of categories (where we only retain the isomorphisms in $\sC$ since $\Add^\rig(\sL,\sC)$ is a groupoid). We have $Z(\sL)=\Z[t]$. Then  $T(\sL)$ has the same universal property in $\Ex^\rig$. 

Describing $T(\sL)$ seems difficult; we restrict to describing $T(\sL_\Q)$ where $\sL_\Q$ is the pseudo-abelian $\otimes$-category obtained by tensoring morphisms with $\Q$ and taking the Karoubian hull: it is universal for $\Q$-linear rigid $\otimes$-categories. For this, we shall use the results of \cite{entovaetal}.

\subsection{The fibres of \eqref{eq2b}} The prime ideals $P$ of $\Q[t]$ may be identified to $t$ as a transcendental number over $\Q$ (for $P=(0)$) and to the algebraic numbers up to conjugation (for the maximal ideals). We write $P_\alpha$ for the  ideal corresponding to such a number $\alpha$, $\Q(\alpha)$ for the quotient field of  $\Q[t]/P_\alpha$ and $\sL_\Q(\alpha)$ for $(\Q(\alpha)\otimes_{\Q[t]}\sL_\Q)^\natural$.

\begin{prop}\label{p22} If $\alpha\notin \Z$, $\sL_\Q(\alpha)$ is abelian semi-simple, hence the fibre of $P_\alpha$ under \eqref{eq2b} consists of a single prime of $\Spec^\otimes \sL_\Q$.
\end{prop}

\begin{proof} This is \cite[Th. 10.5]{dS}.
\end{proof}

\subsection{The case of integers} For $\alpha=n\in \Z$, the situation is more interesting. Here $\sL_\Q(n)$ is the category $\Rep(\GL(n),\Q)$ of \cite[after Def. 10.2]{dS}. We still write $L$ for the image of $L$ in $\sL_\Q(n)$.

For each pair $(p,q)$ of nonnegative integers such that $p-q=n$, with $(p,q)\ne (0,0)$, write $\bP(p|q)$ for the kernel of the $\otimes$-functor 
\begin{equation}\label{eq12}
F(p|q):\sL_\Q(n)\to \Rep_\Q(\GL(p|q)) 
\end{equation}
sending $L$ to $\Q^{p|q}$: since the target category is integral, it is a prime $\otimes$-ideal.  

\begin{lemma}\label{l12} a) We have a chain of inclusions
\begin{align*}
\dots \subset \bP(p+1|q+1)\subset \bP(p|q)\subset\dots \subset \bP(n|0)\subset \sL_\Q(n) &\text{ if } n\ge 0\\
\dots \subset \bP(p+1|q+1)\subset \bP(p|q)\subset\dots \subset \bP(0|-n)\subset \sL_\Q(n) &\text{ if } n\le 0.
\end{align*}
b) We have $\bP(n|0)=\sN$ (\resp $\bP(0|n)=\sN$) if $n> 0$ (\resp if $n< 0$). If $n=0$, we set $\bP(0|0):=\sN$. \\
c) We have a fully faithful $\otimes$-functor 
\[F(\infty|\infty):\sL_\Q(n)\to \Rep_\Q(\GL(\infty|\infty)),\] 
where $\Rep_\Q(\GL(\infty|\infty))\in \Ex^\rig$ is the category constructed in  \cite{entovaetal} and denoted by $\sV_n$ in \loccit\\
d) The $\bP(p|q)$ are the only nonzero $\otimes$-ideals of $\sL_\Q(n)$.\\
e) The space $\Spec^\otimes \sL_\Q(n)$ is homeomorphic to $\N\cup \{\infty\}$ provided with the ``right order topology'' (whose closed subsets are $\N\cup \{\infty\}$ and the intervals $[0,r]$) for the map
\begin{align*}
\bM: \N\cup \{\infty\}&\to \Spec^\otimes \sL_\Q(n)\\
r&\mapsto 
\begin{cases}
\bP(n+r|r) &\text{ if } n\ge 0\\
\bP(r|-n+r) &\text{ if } n\le 0.
\end{cases}
\end{align*} 
The associated constructible topology $A(\N)$ is the Alexandrov compactification of the discrete space $\N$.
\end{lemma}

\begin{proof} a) follows from the existence of a $\otimes$-functor $\Rep_\Q(\GL(p+1|q+1))\to \Rep_\Q(\GL(p|q))$ sending $\Q^{p+1|q+1}$ to $\Q^{p|q}$ (the Duflo-Serganova construction, \cite[\S 7.1]{entovaetal}); or see the reasoning to obtain \cite[(5.1)]{os}. b) follows from \cite[Th. 10.4]{dS}.
c) is  \cite[Prop. 8.1.2]{entovaetal}. 
d) is \cite[Lemma 5.2]{os}. The first statement of e) follows immediately from d), and the second one is obvious.
\end{proof}


\begin{rk}\label{r5} Although Lemma \ref{l12} shows that $\sL_\Q(n)$ is ``$\otimes$-Noetherian'', it also shows that the analogue of Krull's intersection theorem fails completely in this category. Namely, $\bI = \bI^{\otimes n}$ for any $\otimes$-ideal $\bI$ of $\sL_\Q(n)$ and any $n>0$. Indeed, all such ideals are prime, hence $\sL_\Q(n)/\bI^{\otimes n}$ is integral. But $f^{\otimes n}=0$ for any $f\in \bI/\bI^{\otimes n}$.
\end{rk}

Since all $\otimes$-categories $\Rep_\Q(\GL(p|q))$ (including for $(p|q)=(\infty|\infty)$) are abelian and rigid, the functors $F(p|q)$ of \eqref{eq12} and Lemma \ref{l12} c) extend canonically to exact $\otimes$-functors
\[\bar F(p|q):T(\sL_\Q(n))\to \Rep_\Q(\GL(p|q)).\]

\begin{prop}\label{p20} a) The map \eqref{eq2a} is a homeomorphism for the constructible topology on $\Spec^\otimes \sL_\Q(n)$ and yields a ring isomorphism
\[\theta:Z(T(\sL_\Q(n)))\iso \Cont(A(\N),\Q).\]
For $r\in A(\N)$, let $\theta^*(r)=\{z\in Z(T(\sL_\Q(n)))\mid \theta(z)(r)=0\}$. Then the $\otimes$-category $T(\sL_\Q(n))\sslash \theta^*(r)$ is $\Rep_\Q(\GL(n+r|r))$ if $n> 0$, $\Rep_\Q(\GL(r|-n+r))$ if $n< 0$, or $\Rep_\Q(\GL(r+1|r+1))$ if $n=0$. (This includes the case $r=\infty$, the accumulation point of $A(\N)$.)\\
b) Let $A\in T(\sL_\Q(n))$. If $\bar F(\infty|\infty)(A)\ne 0$, then $\bar F(p|q)(A)\ne 0$ for all but a finite number of $(p|q)$. The converse is not true.\\
c) Let $S$ be a finite subset of $\N$. Then there is a canonical decomposition
\[T(\sL_\Q(n))\simeq T(S)\times \prod_{s\in S} T(\sL_\Q(n))\sslash\theta^*(s).\]
In particular, the obvious functor
\[T(\sL_\Q(n)/\bM(r))\to \prod_{s\le r}T(\sL_\Q(n))\sslash\theta^*(s)\]
is an equivalence of $\otimes$-categories for any $r\in \N$.
\end{prop}

\begin{proof} a) We first show the bijectivity of  \eqref{eq2a} together with the last point.  This is a mere translation of \cite[Th. 2]{entovaetal}: indeed, let  $F:T(\sL_\Q(n))\to \sB$ be an exact $\otimes$-functor, with $\sB\in \Ex^\rig$ connected. Let $\Sigma$ be the set of Schur functors killing $F(L)$. If $\Sigma=\emptyset$ (\resp $\Sigma \neq \emptyset$), $F\circ \lambda_{\sL_\Q(n)}$ factors through $F(\infty|\infty)$ (\resp through a unique $F(p|q)$) by part (a) (\resp (b)) of \cite[Th. 2]{entovaetal}; hence $F$ factors through the corresponding $\bar F(p|q)$.  
We now conclude with Corollary \ref{c2}. 

As in the proof of Theorem \ref{c4}, the bijectivity of  \eqref{eq2a} implies that it is a homeomorphism. The second point now follows from Lemma \ref{l13}.

b) By Theorem \ref{p4}, $\bar F(\infty|\infty)(A)\ne 0$ $\iff$ $\theta^*(\infty)\notin V(e(A))$. But the closed subsets of $A(\N)$ not containing $\infty$ are finite. Therefore, $e(A)\in \theta^*(r)$ only for finitely many $r<\infty$, and $\bar F(p|q)(A)=0$ only for the corresponding $(p|q)$. The converse fails because any subset of $A(\N)$ containing $\infty$ is closed.

c) Let $e(S)\in Z(\sL_\Q(n))$ be the idempotent corresponding via a) to the function with value $1$ at every $s\in S$ and $0$ elsewhere. The decomposition of the statement is the one defined by $e$, with $T(S)$ corresponding to $\Ker e$. The description of the other factor follows from Theorem \ref{p14} and Corollary \ref{c8}. The second statement is the special case $S=[0,r]$.
\end{proof}

Here is a complement in the case $n=0$. Let $\sC\in \Add^\rig$ be the category of \cite[5.8]{dtens}: it is $\otimes$-generated by a self-dual object $X$ of Euler characteristic $0$ possessing an endomorphism $\theta$ of trace $1$ and square $0$.   By the universal property, there is a unique $\otimes$-functor $F:\sL_\Q(0)\to \sC$ sending the generator $L$ of $\sL_\Q(0)$ to $X$.

\begin{prop}\label{p24} a)  The functor $F$ is faithful, but not full.\\
b) We have $\sL_\Q(0)/\sN = \Q$ (the category with one object having endomorphisms $\Q$).
\end{prop}

\begin{proof} a) $F$ is  not full because  $\theta\notin \IM F$ since $\End_{\sL_\Q(0)}(L)$ is generated by $1_{L}$. In view of Lemma \ref{l12} d), to prove that it is faithful it suffices to show that $X$ is not Schur-finite, and for this it suffices to show that $S(\theta)\ne 0$ for any Schur functor $S$. Let $r>0$ and let $\sigma\in \fS_r$. Since $\theta^2=0$, we have by \cite[Cor. 7.2.2]{AK2}:
\[\tr(\sigma^{-1}\circ \theta^{\otimes r})=
\begin{cases}
1 &\text{if $\sigma=1$}\\
0 &\text{if $\sigma\ne 1$.}
\end{cases} 
\]

Let $\lambda$ be a partition of $r$, and let $c_\lambda\in \Q[\fS_r]$ be the corresponding Young symmetriser \cite[(4.2)]{FH}.  Up to a nonzero scalar, $\tr(S_\lambda(\theta))$ is the coefficient of $\sigma=1$ in $c_\lambda$. This coefficient is $1$, hence $S_\lambda(\theta)\ne 0$ as requested.

b) This is because $1_{L}\in \sN$.
\end{proof}

\begin{rk} The argument of a) shows, more generally, that there is no $\otimes$-functor from $\sC$ to a Schur-finite category. (See also remark \ref{r4}.)
\end{rk}

\subsection{Globalisation} Recall that $Z(\sL_\Q)=\Q[t]$. By Proposition \ref{p1}, there is a canonical homomorphism 
\begin{equation*}
\Q[t]^\abs\to Z(T(\sL_\Q))
\end{equation*}
where $\Q[t]^\abs$ is the flat completion of $\Q[t]$ \cite{olivier} (recall that $\Spec \Q[t]^\abs\to \Spec \Q[t]$ is bijective on the underlying sets). For any $\alpha$, it fits in a cocartesian square
\begin{equation}\label{eq3}
\begin{CD}
\Q[t]^\abs@>>> Z(T(\sL_\Q))\\
@VVV @VVV\\
\Q(\alpha) @>>> Z(T(\sL_\Q(\alpha))
\end{CD}
\end{equation}

Collecting Propositions \ref{p22} and \ref{p20}, we get:

\begin{thm}\label{t6}  The bottom horizontal homomorphism of \eqref{eq3} is an isomorphism if $\alpha\notin \Z$, and is the `constant' homomorphism $\Q\to \Cont(A(\N),\Q)$ if $\alpha\in \Z$.\qed
\end{thm}

Unfortunately, this does not quite give an explicit description of $Z(T(\sL_\Q))$ as a $\Q[t]^\abs$-algebra, nor \emph{a fortiori} of the category $T(\sL_\Q)$. Nevertheless, it gives a qualitative description of the map $\Spec Z(T(\sL_\Q))\allowbreak \to \Spec \Q[t]^\abs$: it is an isomorphism away from the integers, where the fibres are all isomorphic to $A(\N)$.

\end{document}